\documentclass[12pt]{amsart}
\usepackage{verbatim, latexsym, amssymb, amsmath,color, mathabx}
\usepackage{enumitem}
\usepackage{epsfig}

\usepackage{geometry}
\usepackage{hyperref}

\newtheorem{thm}{Theorem}[section]
\newtheorem{prop}[thm]{Proposition}
\newtheorem{lem}[thm]{Lemma}

\newtheorem{cor}[thm]{Corollary}
\newtheorem{rem}[thm]{Remark}
\newtheorem{exe}{Example}[section]

\DeclareMathOperator{\bR}{\mathbb{R}}
\DeclareMathOperator{\bZ}{\mathbb{Z}}

\DeclareMathOperator{\bfA}{\mathbf{A}}
\DeclareMathOperator{\bfS}{\mathbf{S}}

\DeclareMathOperator{\cA}{\mathcal{A}}

\DeclareMathOperator{\cH}{\mathcal{H}}

\DeclareMathOperator{\cO}{\mathcal{O}}
\DeclareMathOperator{\cP}{\mathcal{P}}

\DeclareMathOperator{\Diff}{\text{Diff}}

\DeclareMathOperator{\TWPFH}{\text{TwPFH}}
\DeclareMathOperator{\TWPFC}{\text{TwPFC}}

\begin{document}
\title[Generic equidistribution]{Generic equidistribution of periodic orbits for area-preserving surface maps}
\author{Rohil Prasad} 
\begin{abstract} We prove that a $C^\infty$-generic area-preserving diffeomorphism of a closed, oriented surface admits a sequence of equidistributed periodic orbits. This is a quantitative refinement of the recently established generic density theorem for area-preserving surface diffeomorphisms. The proof has two ingredients. The first is a ``Weyl law'' for PFH spectral invariants, which was used to prove the generic density theorem. The second is a variational argument inspired by the work of Marques--Neves--Song and Irie on equidistribution results for minimal hypersurfaces and three-dimensional Reeb flows, respectively. 
\end{abstract}

\maketitle

\section{Introduction} \label{sec:intro}

\subsection{Statement of main results} A fundamental question concerning a dynamical system is to describe in as much detail as possible its set of periodic orbits. This work pursues this line of inquiry for area-preserving diffeomorphisms of smooth, closed surfaces. For the remainder of this paper, we fix a closed, smooth, connected, oriented surface $\Sigma$ and an area form $\omega$ on $\Sigma$ of area $1$. We show in this article that for an abundance of area-preserving diffeomorphisms of $\Sigma$, the set of periodic orbits has a rich topological and statistical structure. We find for each of these diffeomorphisms a countable set of periodic orbits which not only has union dense in $\Sigma$, but in some sense evenly fills out the entire surface, a phenomenon known as ``equidistribution''. This is stated formally in the following theorem. 

\begin{thm} \label{thm:generic}
A $C^\infty$-generic area-preserving diffeomorphism of $\Sigma$ has an equidistributed sequence of orbit sets. 
\end{thm}

The second theorem is somewhat finer; it proves a generic equidistribution result in monotone Hamiltonian isotopy classes of area-preserving diffeomorphisms of $\Sigma$. 

\begin{thm}
\label{thm:main} Let $\phi$ be a monotone area-preserving diffeomorphism of $\Sigma$. A $C^\infty$-generic map in the Hamiltonian isotopy class of $\phi$ has an equidistributed sequence of orbit sets. 
\end{thm}

The definition of a monotone area-preserving map can be found in \S\ref{subsec:mappingTorii}. The monotonicity property is ubiquitous; any area-preserving map can be perturbed by a $C^\infty$-small amount to a monotone area-preserving map. Moreover, the identity map is monotone. To provide the reader with more intuition regarding monotone diffeomorphisms, we discuss two examples. 

\begin{exe} \label{ex:monotoneT2} 
Let $\phi$ be any area-preserving diffeomorphism of the two-torus $\bR^2/\bZ^2$, equipped with the standard area form $dx \wedge dy$, which is isotopic to the identity. Then $\phi$ is Hamiltonian isotopic to a translation 
$$(x, y) \mapsto (x + a, y + b)$$
and it is monotone if and only if $a$ and $b$ are rational.
\end{exe} 

Example \ref{ex:monotoneT2} is significant because it implies that, although monotone diffeomorphisms are $C^\infty$-dense in the space of area-preserving maps, they are not necessarily $C^\infty$-generic (the rational numbers are a dense but not generic subset of the reals). This causes a minor complication in that Theorem \ref{thm:generic} above is not an immediate corollary of Theorem \ref{thm:main}. It instead follows from a slight modification of the arguments used to prove Theorem \ref{thm:main}. 

\begin{exe} \label{ex:monotoneHamiltonian}
Any Hamiltonian diffeomorphism of $\Sigma$ is monotone. In general, given two Hamiltonian isotopic area-preserving maps $\phi$ and $\phi'$, the map $\phi$ is monotone if and only if $\phi'$ is. 
\end{exe}

Theorem \ref{thm:main} and Example \ref{ex:monotoneHamiltonian} imply a generic equidistribution result for Hamiltonian diffeomorphisms. 

\begin{cor} \label{cor:hamiltonian}
A $C^\infty$-generic Hamiltonian diffeomorphism of $\Sigma$ has an equidistributed sequence of orbit sets. 
\end{cor}

It is known that the monotone assumption on the map $\phi$ in the statement of Theorem \ref{thm:main} cannot be removed or substantially weakened. Herman \cite{herman} showed that for a Diophantine rotation of the two-torus, any $C^\infty$-small Hamiltonian perturbation is conjugate to the rotation. This produces a (non-monotone) Hamiltonian isotopy class which contains an open set of maps without any periodic orbits, which precludes generic or even dense equidistribution in this isotopy class. This mirrors developments in the recent proof of the generic density theorem for periodic points of area-preserving surface diffeomorphisms; a monotone area-preserving map requires only a small Hamiltonian perturbation to acquire a dense set of periodic points, but a non-monotone map may require a non-Hamiltonian perturbation. 

We now explain what we mean by an ``equidistributed sequence of orbit sets'' in Theorems \ref{thm:generic} and \ref{thm:main}. Terminology surrounding periodic orbits and orbit sets is discussed in more detail in \S\ref{subsec:orbits}. An \textbf{orbit set} of a surface map $\phi$ is a formal sum
$$\cO = \sum_{k=1}^n a_k \cdot S_k$$
where each $a_k$ is a positive real number and each $S_k$ is a simple periodic orbit of $\phi$. Periodic orbits $S$ define functionals on the space $C^\infty(\Sigma)$ of smooth functions on $\Sigma$ by sending a function $f$ to its sum $S(f)$ over the points of the orbit. Similarly, we write
$$\cO(f) = \sum_{k=1}^n a_k \cdot S_k(f)$$
for the sum of $f$ over an orbit set $\cO$. 

We write $|S|$ for the period of a periodic orbit $S$ and correspondingly
$$|\cO| = \sum_{k=1}^n a_k \cdot |S_k|.$$

A sequence $\{\cO_N\}_{N \geq 1}$ of orbit sets is \textbf{equidistributed} if, for any $f \in C^\infty(\Sigma)$, the averages of $f$ over the $\cO_N$ converge to the integral of $f$:
$$\lim_{N \to \infty} \cO_N(f)/|\cO_N| = \int_{\Sigma} f\omega.$$

A formal algebraic argument as in the proof of \cite[Corollary $1.4$]{Irie21} shows that the existence of an equidistributed sequence of orbit sets implies a more concrete type of equidistribution concerning a sequence of genuine periodic orbits. 

\begin{cor} \label{cor:orbits}
There is a $C^\infty$-generic set of area-preserving diffeomorphisms of $\Sigma$ such that any map $\phi$ in this set satisfies the following property. There is a sequence $\{S_N\}_{N \geq 1}$ of periodic orbits of $\phi$ such that, for any $f \in C^\infty(\Sigma)$,
$$\lim_{N \to \infty} \frac{S_1(f) + \ldots + S_N(f)}{|S_1| + \ldots + |S_N|} = \int_\Sigma f\omega.$$
\end{cor}

\begin{rem}
In view of Corollary \ref{cor:hamiltonian}, the exact same result as Corollary \ref{cor:orbits} holds with ``area-preserving'' replaced by ``Hamiltonian''. 
\end{rem}

As mentioned at the beginning of this article, Theorem \ref{thm:generic} shows that, generically, the periodic points of area-preserving surface diffeomorphisms have a rich structure. This potentially has further applications to the study of periodic points of area-preserving surface diffeomorphisms. Irie's equidistribution theorem \cite{Irie21} for three-dimensional Reeb flows, a close cousin of Theorem \ref{thm:generic}, has already had interesting applications in this setting (see \cite{actionLinking, CDHR22}). 

\subsection{Outline of proofs} Theorems \ref{thm:generic} and \ref{thm:main} are proved using spectral invariants arising from Periodic Floer homology (PFH). These are quantitative invariants for area-preserving surface diffeomorphisms originally discovered by Hutchings. They have had numerous recent applications to the study of area-preserving surface maps, including the simplicity conjecture \cite{simplicity}, the large scale geometry of Hofer's metric \cite{largeScale}, and the generic density of periodic points \cite{closingLemma21, edtmairHutchings}. 

The paper \cite{closingLemma21}, beyond the proof of the generic density theorem, proved a broad ``Weyl law'' for PFH spectral invariants\footnote{A Weyl law for PFH spectral invariants corresponding to ``$U$-cyclic'' classes was also proved in \cite{edtmairHutchings}; see the introduction of \cite{closingLemma21} for a discussion regarding the U-cyclic condition and the Weyl laws in \cite{closingLemma21} and \cite{edtmairHutchings}.}, which is stated in \S\ref{subsec:spectralInvariants}.  This is one of the two main ingredients in the proof of Theorems \ref{thm:generic} and \ref{thm:main}. The Weyl law gives an asymptotic relation between the PFH spectral invariants of an area-preserving map $\phi$ and its composition $\phi^H = \phi \circ \phi^1_H$ with the time-one map of a Hamiltonian $H$; the asymptotic difference between the spectral invariants recovers the integral of $H$ over the mapping torus of $\phi$.

Fix a monotone area-preserving map $\phi$. Pairing the Weyl law with a variational argument (the second main ingredient) allows us to find, for any finite set of smooth functions $\{f_1, \ldots, f_k\}$ on $\Sigma$, a small Hamiltonian perturbation $\phi'$ of a map $\phi$ and an orbit set of $\phi'$ such that the average of each $f_i$ over the orbit set is very close to the integral of $f$ over $\Sigma$. A formal argument then produces a nearby Hamiltonian perturbation with a set of \emph{equidistributed} orbit sets, which proves Theorem \ref{thm:main}. This variational argument is inspired by arguments used by Marques--Neves--Song \cite{MarquesNevesSong} and Irie \cite{Irie21} to establish equidistribution results for minimal hypersurfaces and Reeb flows, respectively. Theorem \ref{thm:generic} is subsequently proved by a slight adaptation of the argument for Theorem \ref{thm:main}. 

\subsection{Organization of the paper} The rest of the paper is organized as follows. \S\ref{sec:prelim} goes over several preliminary notions, among them the Weyl law from \cite{closingLemma21} and other properties of PFH spectral invariants used in the later proofs. \S\ref{sec:lemmas} lists several supporting lemmas. \S\ref{sec:thmProof} proves Theorems \ref{thm:generic} and \ref{thm:main}, assuming the results in \S\ref{sec:lemmas}. \S\ref{sec:supportingProofs} provides proofs of the results in \S\ref{sec:lemmas}. 

\subsection{Acknowledgements} This project arose from my collaboration with Dan Cristofaro-Gardiner and Boyu Zhang on the Weyl law for PFH spectral invariants and the generic density theorem for area-preserving maps. I would like to thank them for numerous enlightening conversations. 

After the first version of this work was posted to the arXiv, I was notified of independent work of Edtmair--Yao \cite{edtmairYao21} proving Theorems \ref{thm:generic} and \ref{thm:main}, which also uses a Weyl law, and a non-vanishing result building on \cite[Theorem $1.6$]{closingLemma21}. 

This work was completed under the support of NSF Award \#DGE-1656466.

\section{Preliminaries}\label{sec:prelim}

\subsection{Area-preserving and Hamiltonian diffeomorphisms}\label{subsec:maps}

Denote by $\Diff(\Sigma, \omega)$ the space of area-preserving diffeomorphisms $\phi: \Sigma \to \Sigma$.

Write $C^\infty(\bR/\bZ \times \Sigma)$ for the space of smooth functions on $\bR/\bZ \times \Sigma$. For any $H \in C^\infty(\bR/\bZ \times \Sigma)$ and $t \in \bR/\bZ$, we will use $H_t$ to denote its restriction to $\{t\} \times \Sigma$; this is a smooth function on $\Sigma$. Write $\cH(\Sigma)$ for the space of smooth functions $H: \bR/\bZ \times \Sigma \to \bR$ such that $H_t \equiv 0$ near $t = 0$ and $t = 1$. 

In the subsequent arguments, we will work with multi-parameter families of Hamiltonians. For any $N \geq 1$, write $C^\infty([0,1]^N \times \bR/\bZ \times \Sigma)$ for the space of smooth functions on $[0,1]^N \times \bR/\bZ \times \Sigma$. For any $N \geq 1$ and $H \in C^\infty([0,1]^N \times \bR/\bZ \times \Sigma)$ and any $\tau \in [0,1]^N$, we will use $H^\tau$ to denote the function in $C^\infty(\bR/\bZ \times \Sigma)$ given by restricting $H$ to $\{\tau\} \times \bR/\bZ \times \Sigma$. Write $\cH^N(\Sigma)$ for the space of $H \in C^\infty([0,1]^N \times \bR/\bZ \times \Sigma)$ such that $H^\tau \in \cH(\Sigma)$ for every $\tau \in [0,1]^N$. 

We need to restrict to the slightly smaller spaces $\cH(\Sigma)$ and $\cH^N(\Sigma)$ of Hamiltonians for technical reasons. A Hamiltonian which is not constant near $0 \in \bR/\bZ$ will in general not be well-defined as smooth functions on the mapping torus of an arbitrary area-preserving map. Any Hamiltonian $H \in C^\infty(\bR/\bZ \times \Sigma)$ can be transformed into some $H_* \in \cH(\Sigma)$ with the same time-one map by setting
$$H_*(t, x) = \beta'(t)H_*(\beta(t), x)$$
where $\beta: [0,1] \to [0,1]$ is a smooth, increasing function such that $\beta(t) = 0$ near $t = 0$ and $\beta(t) = 1$ near $t = 1$. The same transformation, applied on each element of the family, changes $H \in C^\infty([0,1]^N \times \bR/\bZ \times \Sigma)$ to an element of the space $\cH^N(\Sigma)$. 

For any $\phi \in \Diff(\Sigma, \omega)$ and $H \in \cH(\Sigma)$, we write $\phi^H \in \Diff(\Sigma, \omega)$ to denote the composition $\phi \circ \phi^1_H$ of $\phi$ with the time-one map of $H$. For any $\phi \in \Diff(\Sigma, \omega)$, $H \in \cH^N(\Sigma)$, and $\tau \in [0,1]^N$, for the sake of simplifying notation we will use $\phi^\tau$ to denote $\phi^{H^\tau}$ when $H$ is clear from context.

\subsection{Orbits and orbit sets}\label{subsec:orbits} Fix any $\phi \in \Diff(\Sigma, \omega)$. A \textbf{periodic orbit} of $\phi$ is a finite ordered set of points $S = \{x_1, \ldots, x_d\} \subset \Sigma$ (possibly with multiplicity) such that $\phi(x_i) = x_{i+1}$ for $i = 1, \ldots, d - 1$ and $\phi(x_d) = x_1$. The cardinality $d = |S|$ of $S$ is called the \textbf{period} of $S$. A periodic orbit $S = \{x_1, \ldots, x_d\}$ is \textbf{simple} if all of the $x_i$ are pairwise distinct. Use $\cP(\phi)$ to denote the set of all simple periodic orbits of $\phi$. 

A periodic orbit $S = \{x_1, \ldots, x_d\}$ is \textbf{nondegenerate} if the Poincar\'e return map $D\phi^d(x_1): T_{x_1}\Sigma \to T_{x_1}\Sigma$ does not have $1$ as an eigenvalue. We say that $\phi$ is \textbf{$d$-nondegenerate} if any periodic orbit with period $\leq d$ is nondegenerate. We say that $\phi$ is \textbf{nondegenerate} if any periodic orbit is nondegenerate.

A \textbf{orbit set}, the space of which we denote by $\cP_{\bR}(\phi)$, is a formal finite linear combination of elements of $\cP(\phi)$ (simple periodic orbits) with positive real coefficients. An \textbf{integral orbit set}, the space of which we denote by $\cP_{\bZ}(\phi)$, is an element of $\cP_{\bR}(\phi)$ where all the coefficients are positive integers. For any orbit set
$$\cO = \sum_{k=1}^N a_k \cdot S_k \in \cP_{\bR}(\phi),$$
 we write
$$|\cO| = \sum_{k=1}^N a_k \cdot |S_k| \in \bR.$$

Orbits and orbit sets define functionals on the set of real-valued functions on $\Sigma$. Fix any function $f: \Sigma \to \bR$. If $S = \{x_1, \ldots, x_d\}$ is a periodic orbit, then we write
$$S(f) = \sum_{i=1}^d f(x_i)$$
for the sum of $f$ along the points in $S$. If 
$$\cO = \sum_{k=1}^N a_k \cdot S_k \in \cP_{\bR}(\phi)$$
is an orbit set, then we write
$$\cO(f) = \sum_{k=1}^N a_k \cdot S_k(f).$$

A sequence $\{\cO_N\}_{N \geq 1}$ of orbit sets is \textbf{equidistributed} if, for any $f \in C^\infty(\Sigma)$, 
$$\lim_{k \to \infty} \cO_N(f)/|\cO_N| = \int_\Sigma f \omega.$$

\subsection{The mapping torus construction}\label{subsec:mappingTorii} Fix any $\phi \in \Diff(\Sigma, \omega)$. Then the \textbf{mapping torus}, denoted by $M_\phi$, is the quotient of $[0, 1]_t \times \Sigma$ by the identification $(1, p) \sim (0, \phi(p))$. It fibers over the circle, which induces a well-defined degree function on $H_1(M_\phi,\mathbb{Z})$. 

The manifold $M_\phi$ has a canonical vector field $R$ induced by the vector field $\partial_t$ on $[0,1] \times \Sigma$. The periodic orbits of $R$ are in one-to-one correspondence with periodic orbits of $\phi$; denote by $\Theta_S$ the orbit of $R$ corresponding to a periodic orbit $S$ of $\phi$. The orbit $\Theta_S$ is simple if and only if $S$ is simple; considering it as an embedded loop in $M_\phi$ oriented by $R$ produces an integral homology class $[S] \in H_1(M_\phi; \bZ)$. The degree of $[S]$ is equal to the period $|S|$. 

Any orbit set $\cO = \sum_{k=1}^N a_k \cdot S_k$ has an associated sum $\Theta_{\cO} = \sum_{k=1}^N a_k \cdot \Theta_{S_k}$, which is a formal linear combination of simple periodic orbits of $R$ with positive real coefficients. Denote by $[\cO] = \sum_{k = 1}^N a_k \cdot [S_k] \in H_1(M_\phi; \bR)$ the associated homology class. If $\cO \in \cP_{\bZ}(\phi)$ then $[\cO] \in H_1(M_\phi; \bZ)$. 

The mapping torus $M_\phi$ also has a canonical closed two form $\omega_{\varphi}$ induced from $\omega$, a canonical closed one-form $dt$, and a canonical two-plane field $V$ which is the vertical tangent bundle. The two-plane field $V$ admits a nonempty, contractible set of almost-complex structures which are compatible with the restriction of $\omega_\phi$; the choice of any one of these defines a first Chern class $c_1(V) \in H^2(M_\phi; \bZ)$. A homology class $\Gamma \in H_1(M_\phi; \bZ)$ is \textbf{monotone} if there is a constant $\rho \neq 0$ such that 
$$\text{PD}(\Gamma) + 2c_1(V) = -\rho[\omega_\phi].$$ 

The class $\Gamma$ is said to be negative/positive monotone if $\rho$ is negative/positive, respectively. We say that $\phi \in \Diff(\Sigma, \omega)$ is \textbf{monotone} if $M_\phi$ has a monotone class. A quick computation shows that if $\phi$ is monotone, then $M_\phi$ always admits a negative monotone class, and more generally a sequence of negative monotone classes with degrees increasing monotonically, which is important for the Weyl law in Theorem \ref{thm:weylLaw} below. As previously stated in Example \ref{ex:monotoneHamiltonian}, given any two Hamiltonian isotopic maps $\phi$ and $\phi'$, $\phi$ is monotone if and only if $\phi'$ is. 

The mapping tori of two Hamiltonian isotopic elements $\phi$ and $\phi'$ of $\Diff(\Sigma, \omega)$ can be identified given the data of a Hamiltonian $H$ such that $\phi' = \phi \circ \phi^1_H$. Define
$$M_H: M_\phi \to M_{\phi'}$$
as the map induced by the diffeomorphism
$$(t, x) \mapsto (t, (\phi^t_H)^{-1}(x))$$
on $[0,1] \times \Sigma$. 

\subsection{Twisted Periodic Floer homology}\label{subsec:pfh} Twisted Periodic Floer homology (PFH) is a Floer theory for area-preserving diffeomorphisms created by Hutchings, see \cite{HutchingsECH02} and \cite{simplicity}. 

Twisted PFH depends on the following data. Pick a monotone, nondegenerate $\phi \in \Diff(\Sigma, \omega)$ and form the mapping torus $M_\phi$. Pick a negative monotone class $\Gamma \in H_1(M_\phi; \bZ)$ of degree $\geq \max(1, G)$. Pick a \emph{trivialized reference cycle}. This consists of two pieces of data. The first is a formal positive integer linear combination $\Theta_{\text{ref}}$ of loops in $M_\phi$ representing the homology class $\Gamma$. The second is a trivialization of the vertical tangent bundle $V$ over these loops. 

Then the twisted PFH (with $\mathbb{Z}/2$-coefficients) is a $\mathbb{Z}$-graded $\mathbb{Z}/2$-vector space
$$\TWPFH_*(\phi, \Gamma, \Theta_{\text{ref}})$$
which is the homology of a chain complex\footnote{There is an auxiliary choice of an almost-complex structure $J$ on $V$ in the definition of the chain complex which we omit from the notation. The twisted PFH groups for different choices of $J$ are canonically isomorphic and the spectral invariants do not depend on this choice either.}
$$\TWPFC_*(\phi, \Gamma, \Theta_{\text{ref}}).$$

The complex is the $\bZ/2$-vector space of formal sums of pairs $(\cO, W)$. Here $\cO$ is an integral orbit set (an element of $\cP_{\bZ}(\phi)$) such that $[\cO] = \Gamma$ and hyperbolic orbits appear in $\cO$ with multiplicity at most $1$ and $W$ is a class in the homology group $H_2(M_\phi, \Theta_{\cO}, \Theta_{\text{ref}}; \bZ)$ associated to $2$-chains with boundary $\Theta_{\cO} - \Theta_{\text{ref}}$. The differential, whose definition is outside of the scope of this paper, is constructed by counting certain pseudoholomorphic curves in $\bR \times M_\phi$.

As discussed in \cite[\S$2.6$]{closingLemma21}, the definition of twisted PFH can be extended to the case where the map $\phi$ is degenerate, by taking a nondegenerate Hamiltonian perturbation. Fix a negative monotone class $\Gamma$ and a trivialized reference cycle $\Theta_{\text{ref}}$. Fix any Hamiltonian $H \in \cH(\Sigma)$ such that the map $\phi^H = \phi \circ \phi^1_H$ is nondegenerate. Then the groups
$$\TWPFH_*(\phi^H, (M_H)_*(\Gamma), M_H(\Theta_{\text{ref}}))$$
are canonically isomorphic for all possible choices of $H$. We can therefore unambiguously define
$$\TWPFH_*(\phi, \Gamma, \Theta_{\text{ref}})$$
to be any one of the groups above. 

\subsection{PFH spectral invariants}\label{subsec:spectralInvariants} Hutchings observed that twisted PFH (unlike the standard version of PFH) admits natural quantitative invariants, called the \textbf{PFH spectral invariants}. We write down a version of PFH spectral invariants following \cite{simplicity,closingLemma21}. 

We begin by defining the spectral invariants for nondegenerate $\phi$. Fix $\Gamma$ negative monotone and a reference cycle $\Theta_{\text{ref}}$ to define the Floer complex $\TWPFC_*(\phi, \Gamma, \Theta_{\text{ref}})$. The \textbf{PFH action functional} is a real-valued on PFH generators, defined as 
$$\mathbf{A}(\cO, W) = \int_W \omega_\phi$$
where $\omega_\phi$ is the canonical closed two-form on $M_\phi$. For any $L \in \bR$, there is an associated subcomplex
$$\TWPFC_*^L(\phi, \Gamma, \Theta_{\text{ref}})$$
of the twisted PFH complex generated by the pairs $(\cO, W)$ such that $\mathbf{A}(\cO, W) \leq L$. For any nonzero class $\sigma \in \TWPFH_*(\phi, \Gamma, \Theta_{\text{ref}})$, we define the \textbf{PFH spectral invariant}
$$c_\sigma(\phi, \Gamma, \Theta_{\text{ref}})$$
to be infimum of all $L$ such that $\sigma$ is represented by a cycle in $\TWPFC_*^L(\phi, \Gamma, \Theta_{\text{ref}})$. We will not use this definition directly, but rather several properties of the PFH spectral invariants established in \cite{largeScale, closingLemma21}. 

Given any data set $\mathbf{S} = (\phi, \Gamma, \Theta_{\text{ref}})$ with $\phi$ nondegenerate and a Hamiltonian $H \in \cH(\Sigma)$ such that $\phi^H$ is nondegenerate, we can define a spectral invariant for $\phi^H$ by ``pushing forward'' the data set by $H$. For any nonzero class $\sigma \in \TWPFH_*(\phi, \Gamma, \Theta_{\text{ref}})$, we write
$$c_\sigma(H; \mathbf{S}) = c_\sigma(H; \phi, \Gamma, \Theta_{\text{ref}})$$
to denote the spectral invariant
$$c_\sigma(\phi^H, (M_H)_*(\Gamma), M_H(\Theta_{\text{ref}}))$$
where $\sigma$ is interpreted as a PFH class of $\phi^H$ via the canonical isomorphism
$$\TWPFH_*(\phi^H, (M_H)_*(\Gamma), M_H(\Theta_{\text{ref}})) \simeq \TWPFH_*(\phi, \Gamma, \Theta_{\text{ref}}).$$

The following inequality, relating the spectral invariants for a pair of Hamiltonians $H$ and $K$, was shown in \cite{closingLemma21}. 

\begin{prop}(Hofer-Lipschitz continuity, \cite[Proposition $5.2$]{closingLemma21}) \label{prop:hoferContinuity}
Fix a data set $\mathbf{S} = (\phi, \Gamma, \Theta_{\text{ref}})$ such that $\Gamma$ has degree $d$. Fix any two Hamiltonians $H$ and $K$ in $\cH(\Sigma)$ such that $\phi^H$ and $\phi^K$ are nondegenerate. Then the following inequality holds:
$$d\int_{\bR/\bZ} \min (H - K)_t dt \leq c_\sigma(H; \mathbf{S}) - c_\sigma(K; \mathbf{S}) + \int_{\Theta_{\text{ref}}} (H - K) dt \leq d\int_{\bR/\bZ} \max (H - K)_t dt.$$
\end{prop}

A posteriori, Proposition \ref{prop:hoferContinuity} allows us to extend the definition of $c_\sigma(H; \bfS)$ to the case where $\phi^H$ is degenerate by setting it to be the limit of $c_\sigma(H_k; \bfS)$ for any sequence of Hamiltonians $\{H_k\}_{k \geq 1}$ in $\cH(\Sigma)$ such that $\phi^{H_k}$ is nondegenerate for every $k$ and the $H_k$ limit to $H$ in the $C^2$ norm. The same bound as in Proposition \ref{prop:hoferContinuity} holds for the extended spectral invariants. Next, we write down the ``spectrality'' property, which shows that the PFH spectral invariants are equal to actions of PFH generators.

\begin{prop}(Spectrality, \cite[Proposition $5.1$]{closingLemma21}) \label{prop:spectrality}
Fix a data set $\bfS = (\phi, \Gamma, \Theta_{\text{ref}})$. Then for any nonzero class $\sigma \in \TWPFH_*(\phi, \Gamma, \Theta_{\text{ref}})$ and Hamiltonian $H \in \cH(\Sigma)$, there is an integral orbit set $\cO$ of $\phi^H$ with $[\cO] = \Gamma$ and a class $W \in H_2(M_\phi, \Theta_{\cO}, \Theta_{\text{ref}}; \bZ)$ such that
$$\bfA(\cO, W) = c_\sigma(H; \phi, \Gamma, \Theta_{\text{ref}}).$$
\end{prop}

The final property that we will make use of is the ``Weyl law'' for PFH spectral invariants, which gives a relative asymptotic formula for the spectral invariants as the degree of $\Gamma$ increases. Fix some nondegenerate $\phi \in \Diff(\Sigma, \omega)$. Fix a sequence of negative monotone classes $\{\Gamma_m\}_{m \geq 1}$ with respective degrees $d_m$ increasing monotonically to $\infty$ and a sequence of reference cycles $\{\Theta_m\}_{m \geq 1}$. Write $\bfS_m = (\phi, \Gamma_m, \Theta_m)$ for every $m$. 

\begin{thm} (Weyl law, \cite[Theorem $1.5$]{closingLemma21}) \label{thm:weylLaw}
Let $\bfS_m = (\phi, \Gamma_m, \Theta_m)$ be a sequence of data sets as fixed above. Then for any sequence of nonzero classes $\sigma_m \in \TWPFH_*(\phi, \Gamma_m, \Theta_m)$ and Hamiltonian $H \in \cH(\Sigma)$, we have the identity
$$\lim_{m \to \infty} \frac{c_{\sigma_m}(H; \bfS_m) - c_{\sigma_m}(0; \bfS_m) + \int_{\Theta_m} H dt}{d_m} = \int_{M_\phi} H \omega_\phi \wedge dt.$$
\end{thm}

The Weyl law in \cite{closingLemma21} is more general than the statement given above. For example, it allows one to compare the asymptotics of spectral invariants along arbitrary pairs of sequences $\{\sigma_m\}_{m \geq 1}$ and $\{\tau_m\}_{m \geq 1}$ of PFH classes for $\phi$ and $\phi^H$, respectively, at the cost of adding additional terms to the asymptotic formula depending on the $\bZ$-gradings of the PFH classes. We conclude the section by noting that the Weyl law requires the twisted PFH groups $\TWPFH_*(\phi, \Gamma_m, \Theta_m)$ to not vanish for all $m \gg 1$. This is guaranteed by the non-vanishing theorem proved in \cite{closingLemma21}. 

\begin{thm} \cite[Theorem $1.6$]{closingLemma21} \label{thm:nonVanishing}
Let $\phi \in \Diff(\Sigma, \omega)$ be a monotone, nondegenerate area-preserving map. Then for any class $\Gamma$ of sufficiently high degree, depending only on the Hamiltonian isotopy class of $\phi$, and any reference cycle $\Theta_{\text{ref}}$, the group $\TWPFH_*(\phi, \Gamma, \Theta_{\text{ref}})$ does not vanish. 
\end{thm}

\section{Supporting lemmas}\label{sec:lemmas} In this section, we state several supporting lemmas used to prove Theorem \ref{thm:main}. Proofs are given in \S\ref{sec:supportingProofs}

\subsection{Generic nondegeneracy of periodic orbits} The first two lemmas concern nondegeneracy of periodic orbits. The first lemma shows that any finite collection of simple periodic orbits of an area-preserving map can be made nondegenerate by an arbitrarily $C^\infty$-small Hamiltonian perturbation. 

\begin{lem}[Nondegenerate after small Hamiltonian perturbation] \label{lem:nonDegenerateAfterSmallPerturbation}
Fix $\phi \in \Diff(\Sigma, \omega)$ and let $S_1, \ldots, S_k$ be a finite set of simple periodic orbits of $\phi$. Then there exist arbitrarily $C^\infty$-small Hamiltonians $H \in C^\infty(\bR/\bZ \times \Sigma)$ such that each of $S_1, \ldots, S_k$ is a nondegenerate simple periodic orbit of the map $\phi^H$.
\end{lem}

The second lemma is a parametric version of the statement that the set of nondegenerate area-preserving maps is generic in each Hamiltonian isotopy class. 

\begin{lem}[Parametric transversality] \label{lem:parametric}
Fix $N \geq 1$ and $\phi \in \Diff(\Sigma, \omega)$. Then for generic $H \in C^\infty([0,1]^N \times \bR/\bZ \times \Sigma)$, 
$$\text{measure}(\{\tau \in [0,1]^N\,|\,\phi^\tau\text{ is nondegenerate}\}) = 1.$$
\end{lem}

\subsection{A fundamental computation} The next lemma addresses the properties of PFH spectral invariants which are differentiable with respect to a finite-dimensional family of Hamiltonian perturbations. 

\begin{lem}[Derivatives of PFH spectral invariants] \label{lem:derivatives}
Fix $N \geq 1$, $\tau^0 \in (0,1)^N$, and a data set $\bfS = (\phi, \Gamma, \Theta_{\text{ref}})$ such that $\phi$ is nondegenerate and $\Gamma$ has degree $d$. Fix a nonzero class $\sigma \in \TWPFH(\phi, \Gamma, \Theta_{\text{ref}})$. Assume that $\phi^{\tau^0}$ is nondegenerate and the function
$$\tau \mapsto c_{\sigma}(H^\tau; \bfS)$$
is differentiable at $\tau^0$. Then there exists an orbit set $\cO \in \cP_{\bZ}(\phi^{\tau^0})$ with $|\cO| = d$ and a class $W \in H_2(M_\phi, \Theta_{\cO}, \Theta_{\text{ref}}; \bZ)$ such that 
$$c_\sigma(H^{\tau^0}; \bfS) = \cA(\cO, W)$$
and for any $i \in \{1, \ldots, N\}$, 
$$\partial_{i}c_\sigma(H^{\tau}; \bfS)(\tau^0) = (\int_{\Theta_{\cO}} - \int_{\Theta_{\text{ref}}}) (\partial_{i} H^\tau(\tau^0) dt).$$
\end{lem}

\subsection{A technical lemma of Marques--Neves--Song} The final lemma is a technical result regarding Lipschitz functions, originally due to Marques--Neves--Song in their work on equidistribution of minimal hypersurfaces in Riemannian manifolds. 

\begin{lem} \label{lem:MNS} \cite{Irie21, MarquesNevesSong}
For any real number $\delta > 0$ and integer $N \geq 1$, there exists $\epsilon = \epsilon(\delta, N) > 0$ such that the following holds. For any Lipschitz function $f$ with $\max(f) - \min(f) \leq 2\epsilon$ and a full measure subset $E \subset [0,1]^N$, there exists $N + 1$ sequences $\{\tau^{1,k}\}_{k\geq 1}, \ldots, \{\tau^{N+1,k}\}_{k\geq1}$ in $E$ satisfying the following conditions:
\begin{itemize}
    \item There exists $\tau^\infty \in (0, 1)^N$ such that 
    $$\lim_{k \to \infty} \tau^{j,k} = \tau^\infty$$
    for any $j = 1, \ldots, N+1$. 
    \item $f$ is differentiable at $\tau^{j,k}$ for any $j$ and any $k$.
    \item For any $j = 1, \ldots, N+1$, the limit
    $$v^j = \lim_{k \to \infty} (\nabla f)(\tau^{j,k}) \in \mathbb{R}^N$$
    exists.
    \item There is a point $v$ in the convex hull of $\{v^1, \ldots, v^{N+1}\}$ of distance at most $\delta$ from the origin. 
\end{itemize}
\end{lem}

\section{Generic equidistribution}\label{sec:thmProof} This section proves Theorems \ref{thm:generic} and \ref{thm:main}.

\subsection{The main propositions} Theorems \ref{thm:generic} and \ref{thm:main} depend on the following propositions. 

\begin{prop} \label{prop:equidistribution}
Fix any monotone $\phi \in \Diff(\Sigma, \omega)$. Fix any integer $N \geq 1$, $\epsilon > 0$, and any finite subset $\{F_i\}_{i = 1}^N$ of $\cH(\Sigma)$. Then for any non-empty open set $U \subset \cH(\Sigma)$, there exists $H \in U$ and an orbit set $\cO \in \cP_{\bR}(\phi^H)$ such that, for any $i \in \{1, \ldots, N\}$,
$$|\frac{1}{|\cO|}\int_{\Theta_{\cO}} F_i dt - \int_{M_\phi} F_i \omega_\phi \wedge dt| < \epsilon.$$
\end{prop}

Proposition \ref{prop:equidistribution} can be summarized as the fact that any monotone map $\phi$ admits a $C^\infty$-small Hamiltonian perturbation which has a ``nearly equidistributed'' orbit set in the mapping torus $M_\phi$. The following proposition, which is a consequence of Proposition \ref{prop:equidistribution}, shows that one can also produce nearly equidistributed orbit sets on the surface $\Sigma$. 

\begin{prop}
\label{prop:equidistribution2} Fix any monotone $\phi \in \Diff(\Sigma, \omega)$. Fix any integer $N \geq 1$, $\epsilon > 0$, and any finite subset $\{f_i\}_{i = 1}^N$ of $C^\infty(\Sigma)$. Then for any non-empty open set $U \subset \cH(\Sigma)$, there exists $H \in U$ and an orbit set $\cO \in \cP_{\bR}(\phi^H)$ such that, for any $i \in \{1, \ldots, N\}$,
$$|\frac{1}{|\cO|}\cO(f_i) - \int_{\Sigma} f_i \omega| < \epsilon.$$
\end{prop}

We defer the proofs of Propositions \ref{prop:equidistribution} and \ref{prop:equidistribution2} to the end of the section.

\subsection{Proof of Theorem \ref{thm:main}} We prove Theorem \ref{thm:main} assuming Proposition \ref{prop:equidistribution2}. Fix any monotone, nondegenerate $\phi \in \Diff(\Sigma, \omega)$. Let $\{f_i\}_{i = 1}^\infty$ be any countable, $C^0$-dense subset of $C^\infty(\Sigma)$. 

For any $N \geq 1$ and $\epsilon > 0$, write $\cH(\Sigma, N, \epsilon)$ for the set of all $H \in \cH(\Sigma)$ such that there is an orbit set $\cO$ of $\phi^H$ such that
\begin{itemize}
    \item Every simple orbit in $\cO$ is nondegenerate. 
    \item For any $i \in \{1, \ldots, N\}$,
    $$|\frac{1}{|\cO|}\cO(f_i) - \int_{\Sigma} f_i \omega| < \epsilon.$$
\end{itemize}

Proposition \ref{prop:equidistribution2}, followed by an application of Lemma \ref{lem:nonDegenerateAfterSmallPerturbation}, ensures that $\cH(\Sigma, N, \epsilon)$ is dense in $\cH(\Sigma)$. It is also open since every simple orbit in $\cO$ is nondegenerate. Now fix some sequence $\{\epsilon_N\}_{N \geq 1}$ converging to zero. Then the set
$$\cH_{\text{good}} = \cap_N \cH(\Sigma, N, \epsilon_N)$$
is generic in $\cH(\Sigma)$. For any $H \in \cH_{\text{good}}$, $\phi^H$ admits a sequence of orbit sets $\{\cO_N\}_{N \geq 1}$ such that, for each $N \geq 1$, and $i \in \{1, \ldots, N\}$,
$$|\frac{1}{|\cO_N|}\cO_N(f_i) - \int_\Sigma f_i\omega| < \epsilon_N.$$

Fix such an $H$ and denote by $\{\cO_N\}_{N \geq 1}$ the associated sequence of orbit sets of $\phi^H$ as defined above. We conclude the proof of Theorem \ref{thm:main} by showing below that this sequence of orbit sets is equidistributed. 

Pick any $f \in C^\infty(\Sigma)$. Write 
$$\delta_N(f) = \inf_{i \in \{1, \ldots, N\}} \|f - f_i\|_{C^0}.$$

The fact that the set $\{f_i\}_{i = 1}^\infty$ is $C^0$-dense implies that, for any fixed $f \in C^\infty(\Sigma)$, $\lim_{N \to \infty} \delta_N(f) = 0$. We estimate for any $N \geq 1$ the difference of the averages of $f$ over $\cO_N$ and over $\Sigma$:
\begin{equation}
    \label{eq:proof1} 
    \begin{split}
    |\frac{1}{|\cO_N|}\cO_N(f)- \int_{\Sigma} f\omega| &\leq \inf_{i \in \{1, \ldots, N\}} \Big(|\frac{1}{|\cO_N|}\cO_N(f_i) - \int_{\Sigma} f_i\omega| \\
    &\qquad + |\frac{1}{|\cO_N|}\cO_N(f - f_i) - \int_{\Sigma} (f - f_i)\omega|\Big) \\
    &\leq \epsilon_N + \inf_{i \in \{1, \ldots, N\}} \|f - f_i\|_{C^0}(|\frac{1}{|\cO_N|}\cO_N(1)| + 1) \\
    &\leq \epsilon_N + 2\delta_N.
    \end{split}
\end{equation}

The last line plugs in the bound for the infimum over $i$ of $\|f - f_i\|_{C^0}$, along with the fact that $\frac{1}{|\cO_N|}\cO_N(1) = 1$. Taking the limit of (\ref{eq:proof1}) as $N \to \infty$ shows equidistribution of the sequence $\{\cO_N\}_{N \geq 1}$: \begin{equation*} \lim_{N \to \infty} |\frac{1}{|\cO_N|}\cO_N(f)- \int_{\Sigma} f\omega| = 0.\end{equation*} 

\subsection{Proof of Theorem \ref{thm:generic}}  Let $\{f_i\}_{i = 1}^\infty$ be any countable, $C^0$-dense subset of $C^\infty(\Sigma)$. To prove Theorem \ref{thm:generic}, we first observe that monotone diffeomorphisms form a $C^\infty$-dense subset of $\Diff(\Sigma, \omega)$ (see \cite[Proposition $5.6$]{closingLemma21}). Write $\mathcal{D}(N, \epsilon) \subset \Diff(\Sigma, \omega)$ for the set of all area-preserving maps which admit an orbit set $\cO$ such that 
\begin{itemize}
    \item Every simple orbit in $\cO$ is nondegenerate. 
    \item For any $i \in \{1, \ldots, N\}$,
    $$|\frac{1}{|\cO|}\cO(f_i) - \int_{\Sigma} f_i \omega| < \epsilon.$$
\end{itemize}

The fact that monotone diffeomorphisms are $C^\infty$-dense, Proposition \ref{prop:equidistribution2}, and Lemma \ref{lem:nonDegenerateAfterSmallPerturbation} ensure that $\mathcal{D}(N, \epsilon)$ is nonempty and dense in $\Diff(\Sigma, \omega)$. The fact that each simple orbit is nondegenerate implies that $\mathcal{D}(N, \epsilon)$ is open as well. Take any sequence $\epsilon_N \to 0$ and write
$$\mathcal{D}_{\text{good}} = \bigcap_N \mathcal{D}(N, \epsilon_N).$$

Then $\mathcal{D}_{\text{good}}$ is $C^\infty$-generic in $\Diff(\Sigma, \omega)$. The argument in the proof of Theorem \ref{thm:main} above shows that each map in $\mathcal{D}_{\text{good}}$ admits an equidistributed sequence of orbit sets, which concludes the proof of Theorem \ref{thm:generic}. 

\subsection{Proof of Proposition \ref{prop:equidistribution}}

We apply the properties of PFH spectral invariants from \S\ref{subsec:spectralInvariants} and the results listed in \S\ref{sec:lemmas} to prove Proposition \ref{prop:equidistribution}. Fix a monotone, nondegenerate $\phi \in \Diff(\Sigma, \omega)$, an integer $N \geq 1$, a parameter $\epsilon > 0$, a nonempty open $U \subset \cH(\Sigma)$, and a subset $\{F_i\}_{i = 1}^N$ of $\cH(\Sigma)$ as in the statement of the proposition. 

\subsubsection{PFH errors are Lipschitz continuous} The fact that $\phi$ is monotone and the nonvanishing result in Theorem \ref{thm:nonVanishing} produces a sequence $\{\Gamma_m\}_{m \geq 1}$ of negative monotone classes with respective degrees $\{d_m\}_{m \geq 1}$ monotonically increasing, a sequence of reference cycles $\{\Theta_m\}_{m \geq 1}$, and a sequence of nonzero classes
$$\sigma_m \in \TWPFH_*(\phi, \Gamma_m, \Theta_m).$$

Fix any $m$ and write $\bfS_m = (\phi, \Gamma_m, \Theta_m)$. We define a real-valued ``error function'' $e_m$ on $\cH(\Sigma)$ by the formula
$$e_m(H) = \int_{M_\phi} H \omega_\phi \wedge dt - \frac{c_{\sigma_m}(H; \bfS_m) - c_{\sigma_m}(0; \bfS_m) + \int_{\Theta_m} H dt}{d_m}.$$

The Weyl law (Theorem \ref{thm:weylLaw}) implies the following crucial identity for any fixed $H \in \cH(\Sigma)$:
$$\lim_{m \to \infty} e_m(H) = 0.$$

We use properties of PFH spectral invariants to show that the error functions are uniformly Lipschitz with respect to the $C^0$ distance on $\cH(\Sigma)$.

\begin{lem}[Uniform Lipschitz continuity] \label{lem:lipschitzErrors} 
For any pair $H$ and $K$ in $\cH(\Sigma)$ and any $m \geq 1$, 
$$|e_m(H) - e_m(K)| \leq 2\|H - K\|_{C^0}.$$
\end{lem}

\begin{proof}
Proposition \ref{prop:hoferContinuity} and the fact that $\omega$ has integral $1$ shows directly that
$$|\frac{c_{\sigma_m}(H; \bfS_m) - c_{\sigma_m}(K; \bfS_m) + \int_{\Theta_m} (H - K)dt}{d_m}| \leq \|H - K\|_{C^0}.$$

The fact that $\omega$ has integral $1$ also shows that 
$$|\int_{M_\phi} (H - K)\omega_\phi \wedge dt| \leq \|H - K\|_{C^0}.$$

Add these two inequalities to deduce the lemma. 
\end{proof}

\subsubsection{Application of Lemma \ref{lem:parametric}} We are free to assume that $U$ has finite $C^2$-diameter. Fix positive constants $c_1$, $c_2$, and $c_3$. The constant $c_1$ will be large, while the constants $c_2$ and $c_3$ will be small. Fix any $K \in U$ and define $F \in \cH^N(\Sigma)$ by setting
$$F^\tau = K + \sum_{i=1}^N \tau_i F_i/c_1.$$

We choose $c_1$ sufficiently large so that $F^\tau \in U$ for every $\tau \in [0,1]^N$. Choose $c_2$ and $c_3$ so that $c_1c_3 + 2c_2 < \epsilon$. By Lemma \ref{lem:parametric}, there is a sequence of elements $\{H_k\}_{k \geq 1}$ in $C^\infty([0,1]^N \times \bR/\bZ \times \Sigma)$ such that the following holds:
\begin{itemize}
    \item The sequence $\{H_k\}_{k \geq 1}$ converges to $F$ in the $C^\infty$ topology.
    \item For every $k \geq 1$, 
    $$\text{measure}(\{\tau \in [0,1]^N\,|\,\phi^{H_k^\tau}\text{ is nondegenerate}\}) = 1.$$
\end{itemize}

We must transform the $H_k$ to elements in $\cH^N(\Sigma)$ which retain the same properties as above. This will be done by the following reparameterization trick. Since $K$ and each function $F_i$ lies in $\cH(\Sigma)$, there exists some $\delta > 0$ such that $F^\tau(t, -) \equiv 0$ for every $i \in \{1, \ldots, N\}$, any $t \in [0, 2\delta) \cup (1 - 2\delta, 1] \subset \bR/\bZ$, and $\tau \in [0,1]^N$. Define a smooth function $\beta: [0,1] \to [0,1]$ satisfying the following properties:
\begin{itemize}
    \item $\beta(t) = 0$ near $t = 0$ and $\beta(t) = 1$ near $t = 1$. 
    \item $\beta'(t) \geq 0$ everywhere.
    \item $\beta(t) = t$ for any $t \in (\delta, 1 - \delta)$.  
\end{itemize}

Define a continuous linear map
$$T_\beta: C^\infty([0,1]^N \times \bR/\bZ \times \Sigma) \to \cH^N(\Sigma)$$
by setting $T_\beta(H)$ to be the map
$$(\tau, t, p) \mapsto \beta'(t)H(\tau, \beta(t), p).$$

As a consequence of the properties of $\beta$, we find $T_\beta(F) = F$. If we fix $k \gg 1$ and set $H = T_\beta(H_k) \in \cH^N(\Sigma)$, we can ensure that $H$ has the following properties:
\begin{itemize}
    \item $H$ lies in the open set $U$. 
    \item $\text{measure}(\{\tau \in [0,1]^N\,|\,\phi^{H^\tau}\text{ is nondegenerate}\}) = 1.$
    \item For every $i \in \{1, \ldots, N\}$, we have the bound $\|c_1\partial_iH^\tau - F_i\|_{C^0} < c_2$. 
\end{itemize}

The first property is a consequence of the fact that $H_k \to F$ in the $C^\infty$ topology, which implies by continuity of $T_\beta$ that $T_\beta(H_k) \to T_\beta(F) = F$. The second property is a consequence of the fact that, for any $\tau \in [0,1]$ and $K \in C^\infty([0,1]^N \times \bR/\bZ \times \Sigma)$, the time-one map of $(T_\beta(K))^\tau$ is the same as the time-one map of $K^\tau$. The third property is a consequence of the fact that $H_k \to F$ in the $C^\infty$ topology and the fact that the operator $T_\beta$ commutes with differentiation in the $\tau$-variable. We know that $c_1\partial_i H_k^\tau \to F_i$ in the $C^0$ topology, which in turn implies that $$c_1\partial_{i} T_\beta(H_k^\tau) = c_1 T_\beta(\partial_i H_k^\tau) \to T_\beta(F_i) = F_i$$
in the $C^0$ topology. 

\subsubsection{Application of Lemmas \ref{lem:derivatives} and \ref{lem:MNS}} Recall the positive constants $c_1$, $c_2$, and $c_3$ and family $H \in \cH^N(\Sigma)$ fixed above. 

Define a sequence of uniformly Lipschitz functions $\overline{e}_m: [0,1]^N \to \mathbb{R}$ by the formula
$$\overline{e}_m(\tau) = e_m(H^\tau).$$

Since they are uniformly Lipschitz functions on a compact space and pointwise converge to zero, the functions $\overline{e}_m$ uniformly converge to zero. Fix $m \gg 1$ so that $\|\overline{e}_m(\tau)\|_{C^0} < \epsilon(c_3, N)$ and apply Lemma \ref{lem:MNS} with the constant $\delta = c_3$ and the full measure set $E$ being the set on which $\phi^{H^\tau}$ is nondegenerate. This produces sequences $\{\tau^{j,k}\}_{k \geq 1}$ for $j \in \{1, \ldots, N+1\}$ in $(0,1)^N$ such that the following conditions hold:
\begin{itemize}
    \item The map $\phi^{\tau^{j,k}}$ is nondegenerate for any $j$ and $k$.
    \item There exists $\tau^\infty \in (0, 1)^N$ such that 
    $$\lim_{k \to \infty} \tau^{j,k} = \tau^\infty$$
    for any $j = 1, \ldots, N+1$. 
    \item $\overline{e}_m$ is differentiable at $\tau^{j,k}$ for any $j$ and any $k$.
    \item For any $j = 1, \ldots, N+1$, the limit
    $$v^j = \lim_{k \to \infty} (\nabla \overline{e}_m)(\tau^{j,k}) \in \mathbb{R}^N$$
    exists.
    \item There is a point $v$ in the convex hull of $\{v^1, \ldots, v^{N+1}\}$ of distance at most $c_3$ from the origin. 
\end{itemize}

The last bullet yields a set of positive real numbers $\{a_j\}_{j = 1}^{N+1}$ such that $\sum_{j=1}^N a_j = 1$ and
$$\|\sum_{j=1}^N a_j v^j\| < c_3.$$

It follows that for all sufficiently large $k$, 
\begin{equation} \label{eq:proof2} \|\sum_{j=1}^N a_j (\nabla \overline{e}_m)(\tau^{j,k})\| < c_3. \end{equation}

Lemma \ref{lem:derivatives} and the first and third bullets above show that for any $j$ and $k$, there is an integral orbit set $\cO_{j,k}$ with $|\cO_{j,k}| = d_m$ and a class $W_{j,k}$ such that
$$c_{\sigma_m}(H^{\tau^{j,k}}; \bfS_m) = \bfA(\cO_{j,k}, W_{j,k})$$
and
\begin{equation} \label{eq:proof3} \partial_i c_{\sigma_m}(H^{\tau}; \bfS_m)(\tau^{j,k}) = (\int_{\Theta_{\cO_{j,k}}} - \int_{\Theta_{\text{ref}}})( (\partial_i H^\tau)(\tau^{j,k}) dt).\end{equation}

Equation (\ref{eq:proof3}) and the fact that the $i$-th derivative of $H^\tau$ is close to $F_i/c_1$ yields the following bound on the $i$-th derivative of $\overline{e}_m$ at any $\tau^{j,k}$:
\begin{equation} \label{eq:proof4} 
\begin{split}
(\partial_i \overline{e}_m)(\tau^{j,k}) &= \big(\int_{M_\phi} (\partial_i H^\tau)(\tau^{j,k}) \omega_\phi \wedge dt -  \frac{1}{d_m}\int_{\Theta_{\cO_{j,k}}}(\partial_i H^\tau)(\tau^{j,k}) dt\big) \\
&\geq c_1^{-1}\big(\int_{M_\phi} F_i\omega_\phi \wedge dt -  \frac{1}{d_m}\int_{\Theta_{\cO_{j,k}}} F_i dt\big) - 2c_1^{-1}c_2 \\
\end{split}
\end{equation}

The second line splits up $\partial_i H^\tau = F_i + (\partial_i H^\tau - F_i)$ and uses the bound $\|c_1\partial_i H^\tau - F_i\|_{C^0} < c_2$. 

Write $\cO_k = \sum_{j=1}^N a_j \cO_{j,k}$, where the $a_j$ are as in (\ref{eq:proof2}). Plug in (\ref{eq:proof4}) into (\ref{eq:proof2}) to deduce the following inequality for any $i \in \{1, \ldots, N\}$ and any sufficiently large $k$:

\begin{equation}
    \label{eq:proof5} |\frac{1}{|\cO_k|}\int_{\Theta_{\cO_k}} F_i dt - \int_{M_\phi} F_i \omega_\phi \wedge dt| \leq c_1c_3 + 2c_2. 
\end{equation}

Observe that $|\cO_k| = d_m$ for every $k$. A compactness argument implies that a subsequence of the $\cO_k$ will converge to an orbit set $\cO$ for $\phi^{\tau^\infty}$ with $|\cO| = d_m$ satisfying
\begin{equation}
    \label{eq:proof6} 
    |\frac{1}{|\cO|}\int_{\Theta_{\cO}} F_i dt - \int_{M_\phi} F_i \omega_\phi \wedge dt| \leq c_1c_3 + 2c_2 < \epsilon
\end{equation}
for any $i \in \{1, \ldots, N\}$. The second inequality in (\ref{eq:proof6}) follows from our a priori choice of constants $c_1$, $c_2$, $c_3$. This concludes the proof of the proposition.

\subsection{Proof of Proposition \ref{prop:equidistribution2}}

We prove Proposition \ref{prop:equidistribution2} using Proposition \ref{prop:equidistribution}. Fix a monotone, nondegenerate $\phi \in \Diff(\Sigma, \omega)$. Fix $N$, $\epsilon$, and $U$ as in the statement of the proposition. 

\subsubsection{Reduction to the case $0 \in U$} It is sufficient to consider the case $0 \in U$ to prove the proposition. Recall the composition operation for Hamiltonians in $\cH(\Sigma)$:
$$(K_1 \# K_2)(t, x) = K_1(t, x) + K_2(t, (\phi^t_{K_1})^{-1}(x)).$$

The Hamiltonian $K_1 \# K_2$ satisfies $\phi^t_{K_1 \# K_2} = \phi^t_{K_1} \circ \phi^t_{K_2}$ for each $t$. For any fixed Hamiltonian $K \in \cH(\Sigma)$, write $\bar{K}(t,x) = -K(t, \phi^t_K(x))$ for its inverse with respect to the operation $\#$. 

Fix any $H \in U$. Then write 
$$U^H = \{\bar{H}\#K \,|\,K \in U\} \subset \cH(\Sigma),$$
which by definition is an open set containing $\bar{H}\#H = 0$. Then Proposition \ref{prop:equidistribution2} is equivalent to the version where we replace the base map $\phi$ by $\phi^H$ and the open set $U$ by $U^H$.

\subsubsection{Application of Proposition \ref{prop:equidistribution}} 
By the previous step, we are free to assume $0 \in U$. Let $\{f_i\}_{i = 1}^{N}$ be the finite set from the statement of the proposition. Fix a non-negative smooth function $\chi: [0,1] \to [0, \infty)$ which is compactly supported in $(0,1)$ and integrates to $1$ over the real line. Define for any $i$ the function
$$G_i(t,x) = \chi(t)f_i(x) \in \cH(\Sigma).$$

Write for any $H \in U$ the function
$$G_i^H(t, x) = \chi(t)f_i( (\phi_H^t)^{-1}(x)) \in \cH(\Sigma).$$

Since $U$ contains $0$, we can assume without loss of generality that it has sufficiently $C^2$-small diameter so that for any $i$ and any $H \in U$,
\begin{equation} \label{eq:prop6} \|G_i - G_i^H\|_{C^0} < \epsilon/6. \end{equation}

Fix any $C^0$-dense set $\{F_j\}_{j=1}^\infty \subset \cH(\Sigma)$. Fix some $N_* \gg 1$ so that for any $i \in \{1, \ldots, N\}$, there is some $j(i) \in \{1, \ldots, N_*\}$ such that
\begin{equation} \label{eq:prop1} \|G_i - F_{j(i)}\|_{C^0} < \epsilon/6. \end{equation}

Proposition \ref{prop:equidistribution} produces $H \in U$ and an orbit set $\cO \in \cP_{\bR}(\phi^H)$ such that for any $j \in \{1, \ldots, N_*\}$, 
\begin{equation} \label{eq:prop2} |\frac{1}{|\cO|}\int_{\Theta_{\cO}} F_j dt - \int_{M_\phi} F_j \omega_\phi \wedge dt| < \epsilon/3. \end{equation}

The functions $G_i^H$ recover data about the functions $f_i$ in the following way. By definition, the one-form $G_i^H dt$ on $M_\phi$ is the pullback $(M_H)^*(G_i dt)$ and the three-form $G_i^H \omega_\phi \wedge dt$ is the pullback $(M_H)^*(G_i \omega_{\phi^H} \wedge dt)$. For any orbit set $\cO$ for $\phi^H$, we then compute
\begin{equation}
    \label{eq:prop3}
    \begin{split}
    \frac{1}{|\cO|}\int_{(M_H)^{-1}(\Theta_{\cO})} G_i^H dt &= \frac{1}{|\cO|}\int_{\Theta_{\cO}} G_i dt \\
    &= \frac{1}{|\cO|} \cO(f_i).
    \end{split}
\end{equation}

The last line follows from lifting to $[0,1] \times \Sigma$. Expand 
$$\cO = \sum_{k=1}^n a_k \cdot S_k = \sum_{k = 1}^n a_k \cdot \{x_{k,1}, \ldots, x_{k, d_k}\}.$$ 

Each orbit $\Theta_{S_k} \subset M_{\phi^H}$ lifts to a union of straight line segments
$$\bigsqcup_{i=1}^{d_k}\,[0,1] \times \{x_{k,i}\} \subset [0,1] \times \Sigma.$$ 

The integral of $G_i dt$ over $[0,1] \times \{x_{k,i}\}$ is the integral of $f_i(x_{k,i})\chi(t) dt$ over this line, which is equal to $f_i(x_{k,i})$ since $\chi(t)$ integrates to $1$. Therefore, the integral of $G_i^H dt$ over $\Theta_{S_k}$ is $S_k(f_i)$; summing up over all $k$ yields the desired result. We also compute
\begin{equation} \label{eq:prop4}
\begin{split}
    \int_{M_\phi} G_i^H \omega_\phi \wedge dt &= \int_{M_{\phi^H}} G_i \omega_{\phi^H} \wedge dt \\
    &= \int_{[0,1] \times \Sigma} \chi(t)f_i(x)\omega \wedge dt \\
    &= \big(\int_{\Sigma} f_i\omega\big) \cdot \big(\int_{0}^1 \chi(t) dt \big) \\
    &= \int_{\Sigma} f_i\omega. 
\end{split}
\end{equation}

We are now sufficiently prepared to conclude the proposition by proving the following bound for any $i \in \{1, \ldots, N\}$:
\begin{equation}
    \label{eq:prop5} 
    \begin{split}
    |\frac{1}{|\cO|}\cO(f_i) - \int_{\Sigma} f_i \omega| &= |\frac{1}{|\cO|}\int_{\Theta_{\cO}} G_i^H dt - \int_{M_\phi} G_i^H \omega \wedge dt| \\
    &\leq |\frac{1}{|\cO|}\int_{\Theta_{\cO}} F_{j(i)} dt - \int_{M_\phi} F_{j(i)} \omega \wedge dt| \\
    &\qquad + \frac{1}{|\cO|}\int_{\Theta_{\cO}} |F_{j(i)} - G_i^H|  dt + \int_{M_\phi} |F_{j(i)} - G_i^H| \omega \wedge dt \\
    &< \epsilon/3 + \frac{1}{|\cO|}\int_{\Theta_{\cO}} |F_{j(i)} - G_i^H|  dt + \int_{M_\phi} |F_{j(i)} - G_i^H| \omega \wedge dt \\
    &< \epsilon/3 + 2\epsilon/3 \\
    &= \epsilon. 
    \end{split}
\end{equation}

The first line uses the computations (\ref{eq:prop3}) and (\ref{eq:prop4}). The third line uses (\ref{eq:prop2}). The fourth line combines (\ref{eq:prop6}) and (\ref{eq:prop1}).

\section{Proofs of supporting lemmas}\label{sec:supportingProofs} In this section, we give proofs of the supporting lemmas from \S\ref{sec:lemmas}.

\subsection{Proof of Lemma \ref{lem:parametric}} Fix a nondegenerate, monotone $\phi \in \Diff(\Sigma, \omega)$ and $N \geq 1$ as in the statement of the lemma. For any $l \geq 0$, write $C^l([0,1]^N \times \bR/\bZ \times \Sigma)$ for the Banach space of functions of class $C^l$ on $[0,1]^N \times \bR/\bZ \times \Sigma$. The proof of Lemma \ref{lem:parametric} closely follows the proof of the analogous result in \cite[\S$4$]{Irie21}. 

\subsubsection{Statement of necessary lemmas} Lemma \ref{lem:parametric} will be a consequence of the following three lemmas. 

\begin{lem}\label{lem:parametric1}
For any $l \geq 3$, there is a generic set of $H \in C^l([0,1]^N \times \bR/\bZ \times \Sigma)$ such that 
$$\text{measure}(\{\tau \in [0,1]^N\,|\,\phi^\tau\text{ is nondegenerate}\}) = 1.$$
\end{lem}

For any $l \geq 3$, define $\mathcal{M}^l$ to be the space of tuples $(H, \tau, S)$ where $H \in C^l([0,1]^N \times \bR/\bZ \times \Sigma)$, $\tau \in [0,1]^N$, and $S$ is a simple periodic orbit of $\phi^\tau$. The space of simple periodic orbits is topologized as a subset of the disjoint union $\bigsqcup_{d \geq 1} \Sigma^d$ of products of $\Sigma$; this gives $\mathcal{M}^l$ a natural topology. For any $d \geq 1$, we denote by $\mathcal{M}^{l,d}$ the connected component of $\mathcal{M}^l$ consisting of tuples $(H, \tau, S)$ with $|S| = d$. 

\begin{lem} \label{lem:parametric2}
For any $l \geq 3$ and $d \geq 1$, the space $\mathcal{M}^{l,d}$ has the structure of a Banach manifold of class $C^{l - 1}$ such that the projection
$$\mathcal{M}^{l,d} \to C^l([0,1]^N \times \bR/\bZ \times \Sigma)$$
is a $C^{l-1}$ Fredholm map of index $N$. 
\end{lem}

For any Hamiltonian $H$ and any periodic orbit $S = \{x_1, \ldots, x_d\}$ of $\phi^H$, denote by 
$$\rho(S, H): T_{x_1}\Sigma \to T_{x_1}\Sigma$$
the return map of the orbit $S$. Note that $\rho(S, H)$ preserves the symplectic form $\omega$ on the vector space $T_{x_1}\Sigma$. It follows that $\phi^H$ is nondegenerate if and only if, for every simple orbit $S$, the return map $\rho(S, H)$ does not have a root of unity as an eigenvalue. Define $\mathcal{M}_{\text{bad}}^l$ to be the set of points $(H, \tau, S) \in \mathcal{M}^l$ for which the return map $\rho(S, H^\tau)$ has a root of unity as an eigenvalue. For any $d \geq 1$, we denote by $\mathcal{M}^{l,d}_{\text{bad}}$ the connected component of $\mathcal{M}^{l}_{\text{bad}}$ consisting of tuples $(H, \tau, S)$ with $|S| = d$. 

\begin{lem}\label{lem:parametric3}
For any $l \geq 3$ and $d \geq 1$, the space $\mathcal{M}^{l,d}_{\text{bad}}$ is a countable union of $C^{l - 2}$ Banach submanifolds of $\mathcal{M}^{l,d}$ of codimension at least $1$. 
\end{lem}

\subsubsection{Proof assuming the lemmas} We prove Lemma \ref{lem:parametric} assuming Lemmas \ref{lem:parametric1}, \ref{lem:parametric2}, and \ref{lem:parametric3}. For any integer $d \geq 1$ and $\delta > 0$, write $E(d, \delta)$ for the set of Hamiltonians in $C^\infty([0,1]^N \times \bR/\bZ \times \Sigma)$ such that, for any $H \in E(d, \delta)$, 
\begin{equation} \label{eq:parametric1} \text{measure}(\{\tau \in [0,1]^N\,|\,\text{$\phi^\tau$ is $d$-nondegenerate}\}) > 1 - \delta. \end{equation}

We claim that $E(d, \delta)$ is open and dense in $C^\infty([0,1]^N \times \bR/\bZ \times \Sigma)$. It is clearly open. To show denseness, we use Lemma \ref{lem:parametric1}. For any $l \geq 3$, Lemma \ref{lem:parametric1} implies that the set of $H \in C^l([0,1]^N \times \bR/\bZ \times \Sigma)$ satisfying (\ref{eq:parametric1}) is open and dense. Since $C^\infty([0,1]^N \times \bR/\bZ \times \Sigma)$ is dense in $C^l([0,1]^N \times \bR/\bZ \times \Sigma)$ for every $l \geq 3$, we conclude that $E(d, \delta)$ is dense in $C^l([0,1]^N \times \bR/\bZ \times \Sigma)$ for every $l \geq 3$. It follows that it is dense in $C^\infty([0,1]^N \times \bR/\bZ \times \Sigma)$. 

Fix a sequence $d_N \to \infty$ and a sequence $\delta_N \to 0$. Then the set
$$E = \bigcap_N E(d_N, \delta_N)$$
is the generic set desired by Lemma \ref{lem:parametric}. 

\subsubsection{Proof of Lemma \ref{lem:parametric1}} We prove Lemma \ref{lem:parametric1} assuming Lemmas \ref{lem:parametric2} and \ref{lem:parametric3}. Fix $l \geq 3$ and $d \geq 1$. Introduce the notation 
$$\Pi: \mathcal{M}^{l,d} \to C^l([0,1]^N \times \bR/\bZ \times \Sigma)$$
for the projection.

Lemmas \ref{lem:parametric2} and \ref{lem:parametric3}, along with the Sard-Smale theorem implies that there is a generic set $E_d \subset C^l([0,1]^N \times \bR/\bZ \times \Sigma)$ satisfying the following properties:
\begin{itemize}
    \item For any $H \in E_d$, the preimage $\Pi^{-1}(H)$ is a manifold of class $C^{l-1}$ with dimension $N$. 
    \item For any $H$ in $E_d$, the preimage $\Pi^{-1}(H) \cap \mathcal{M}^{l,d}_{\text{bad}}$ is a countable union of submanifolds of class $C^{l-2}$ of $\Pi^{-1}(H)$ with codimension at least $1$. 
\end{itemize}

Fix any $H \in E_d$. Sard's theorem implies that there is a full measure set of $\tau \in [0,1]$ which are simultaneously regular values of the projections $\Pi^{-1}(H) \to [0,1]^N$ and $\Pi^{-1}(H) \cap \mathcal{M}^{l,d}_{\text{bad}} \to [0,1]^N$. For each regular value $\tau$, the preimage of $\tau$ under this projection does not intersect $\Pi^{-1}(H) \cap \mathcal{M}^{l,d}_{\text{bad}}$ since it is a union of submanifolds of dimension $\leq N - 1$. It follows by definition that, for any $H \in E_d$, $\phi^\tau$ is nondegenerate for a full measure set of $\tau \in [0,1]^N$. Taking the intersection of the sets $E_d$ across all $d$ yields the generic set desired by Lemma \ref{lem:parametric1}. 

\subsubsection{Proof of Lemma \ref{lem:parametric2}} Fix $l \geq 3$ and $d \geq 1$. We give $\mathcal{M}^{l,d}$ the structure of a Banach manifold using the Banach manifold implicit function theorem. Let $\Delta^d \subset \Sigma^d$ denote the ``thick diagonal'' consisting of tuples of points $\{x_1, \ldots, x_d\}$ such that $x_i = x_j$ for some $i \neq j$. Define a map
$$\Psi: C^l([0,1]^N \times \bR/\bZ \times \Sigma) \times [0,1]^N \times (\Sigma^d \setminus \Delta^d) \to (\Sigma^d)^2$$
sending a tuple $(H, \tau, S = \{x_1, \ldots, x_d\}$) to the pair of sets
$$(\{x_1, \ldots, x_d\}, \{\phi^\tau(x_d), \phi^\tau(x_1), \ldots, \phi^\tau(x_{d-1})\}.$$

Write $Z \subset (\Sigma^d)^2$ for the diagonal. Then by definition, $\mathcal{M}^{l,d} = \Psi^{-1}(Z)$. The linearization $D\Psi$ sends a variation $h$ in the $H$ direction at a point $(H, \tau, S = \{x_1, \ldots, x_d\}) \in \mathcal{M}^{l,d}$ to the pair 
$$(\{0, \ldots, 0\}, \{V_h(x_d), V_h(x_1), \ldots, V_h(x_{d-1})\}) \in \Big(\bigoplus_{i=1}^d T_{x_i}\Sigma\Big)^2$$
where $V_h = \partial_t|_{t = 0} (\phi \circ \phi^1_{H^\tau + th^\tau})$. We can choose $h$ so that $$\{V_h(x_d), V_h(x_1), \ldots, V_h(x_{d-1})\} \in \bigoplus_{i=1}^d T_{x_i}\Sigma$$ is any tuple of tangent vectors. It follows that the map $\Psi$ is transverse to the diagonal $Z$. The implicit function theorem for Banach manifolds implies that $\mathcal{M}^{l,d}$ is a Banach submanifold of class $C^{l-1}$ and codimension $2d$ in $C^l([0,1]^N \times \bR/\bZ \times \Sigma) \times [0,1]^N \times (\Sigma^d \setminus \Delta^d)$.

The projection
$$C^l([0,1]^N \times \bR/\bZ \times \Sigma) \times [0,1]^N \times (\Sigma^d \setminus \Delta^d) \to C^l([0,1]^N \times \bR/\bZ \times \Sigma)$$
is Fredholm and $C^l$ of index $N + 2d$, from which we conclude that the projection
$$\mathcal{M}^{l,d} \to C^l([0,1]^N \times \bR/\bZ \times \Sigma)$$
is Fredholm and $C^{l-1}$ of index $N$. 

\subsubsection{Proof of Lemma \ref{lem:parametric3}} \label{subsubsec:parametric3} Fix $l \geq 3$ and $d \geq 1$. We define a universal fiber bundle $\mathcal{E} \to \mathcal{M}^{l,d}$ of class $C^{l-1}$ by setting its fiber at $(H, \tau, S = \{x_1, \ldots, x_d\})$ to equal the Lie group of linear symplectic automorphisms of $T_{x_1}\Sigma$. 

For any root of unity $\zeta$, define a fiber bundle $\mathcal{E}_{\zeta} \subset \mathcal{E}$ whose fibers consist of linear symplectic automorphisms which have $\zeta$ as an eigenvalue. It follows that for any $\zeta$, $\mathcal{E}_{\zeta}$ is a finite union of sub-bundles of $\mathcal{E}$, each with fibers given by submanifolds of codimension at least $1$. When $\zeta \not\in \{1, -1\}$, $\mathcal{E}_\zeta$ is a genuine codimension $1$ subbundle. 

There is a natural $C^{l-1}$ section $\alpha: \mathcal{M}^{l,d} \to \mathcal{E}$ taking $(H, \tau, S)$ to the return map $\rho(S, H^\tau)$. The lemma follows from showing that $\alpha$ is transverse to each of the $\mathcal{E}_{\zeta}$; it suffices to consider variations in the $H$ factor to achieve this. 

Fix any $(H, \tau, S = \{x_1, \ldots, x_d\})$ and let $M \in \mathfrak{sp}(T_{x_1}\Sigma)$ be any element of the Lie algebra of $\mathcal{E}(H, \tau, S)$. It is straightforward to construct a one-parameter family $(H^{M,s})_{s \in [0,1]}$ of Hamiltonians satisfying the following properties:
\begin{itemize}
    \item $H^{M,0} = H^\tau$.
    \item For any $s \in [0,1]$, $H^{M,s} - H^\tau$ is supported in a small neighborhood of $x_1$ disjoint from any other point in $S$ and vanishes to first order on $x_1$. 
    \item For any $s \in [0,1]$, $\rho(S, H^{M,s}) = \rho(S, H^\tau) \circ \text{exp}(sM)$. 
\end{itemize}

The construction proceeds roughly as follows. Take local symplectic coordinates centered at $x_1$. There is a smooth family of Hamiltonians $(H^s)_{s \in [0,1]}$ which are supported in a neighborhood of $x_1$ disjoint from any other point in $S$ and have time-one maps which equal the linear symplectic map $\text{exp}(sM)$ near $x_1$ in the symplectic coordinates. We can also assume $H^0 \equiv 0$. Letting 
$$(K_1 \# K_2)(t, x) = K_1(t, x) + K_2(t, (\phi^t_{K_1})^{-1}(x))$$
be the composition operation for Hamiltonians, we set $H^{M,s} = H^\tau \# H^s$. Then by definition of the composition operation, the time-one map of $H^{M,s}$ is equal to $\phi^1_{H^\tau} \circ \phi^1_{H^s}$. As a result, the family $(H^{M,s})_{s \in [0,1]}$ satisfies the properties above. 

Applying the construction above for all $M$ not contained in the tangent space to $\mathcal{E}_\zeta(H, \tau, S) \subset \mathcal{E}(H, \tau, S)$ at $\alpha(H, \tau, S)$ implies that $\alpha$ is transverse to $\mathcal{E}_\zeta$ as required.  

\subsection{Proof of Lemma \ref{lem:nonDegenerateAfterSmallPerturbation}} Lemma \ref{lem:nonDegenerateAfterSmallPerturbation} follows from the same construction as the one used to prove Lemma \ref{lem:parametric3} above. For any Hamiltonian $H$, an orbit $S = \{x_1, \ldots, x_d\}$, and $M \in \mathfrak{sp}(T_{x_1}\Sigma)$, there is a smooth family of Hamiltonians $(H^{M,s})_{s \in [0,1]}$ such that $H = H^{M,s}$, $H^{M,s} - H$ is compactly supported in a neighborhood of $x_1$, equals $0$ on $x_1$, and the return maps satisfy 
$$\rho(S, H^{M,s}) = \rho(S, H) \circ \text{exp}(sM).$$

It follows with an appropriate choice of $M$ and $s \ll 1$, we can construct arbitrarily $C^\infty$-small perturbations of $H$, supported near $S$, whose time-one maps have $S$ as a nondegenerate simple periodic orbit. It follows that there are $C^\infty$-small perturbations making any finite collection of simple periodic orbits nondegenerate. This proves Lemma \ref{lem:nonDegenerateAfterSmallPerturbation}. 

\subsection{Proof of Lemma \ref{lem:derivatives}} Fix $N \geq 1$, $\tau^0 \in (0,1)^N$, the data set $\bfS = (\phi, \Gamma, \Theta_{\text{ref}})$, and $\sigma \in \TWPFH(\phi, \Gamma, \Theta_{\text{ref}})$ as in the statement of the lemma. Let $d$ denote the degree of the class $\Gamma$.

For any $\tau \in [0,1]^N$ and any orbit set $\cO \in \cP_{\bR}(\phi^\tau)$, we pull back by the map $M_{H^\tau}$ to think of $\Theta_{\cO}$ as a set of embedded loops with multiplicity in $M_\phi$; each loop is a simple periodic orbit of the vector field $R + X^{H^\tau}$ where $R$ is the canonical vector field on $M_\phi$ and $X^{H^\tau}$ is the time-dependent Hamiltonian vector field of $H^\tau$, considered as a vector field on $M_\phi$ in the natural way. A class $W \in H_2(M_{\phi^\tau}, \Theta_{\cO}, M_{H^\tau}(\Theta_{\text{ref}}); \bZ)$ can be thought of as a class in $H_2(M_\phi, \Theta_{\cO}, \Theta_{\text{ref}}; \bZ)$. We also note that the pullback of the two-form $\omega_{\phi^\tau}$ by $M_{H^\tau}$ is equal to $\omega_\phi + dH^\tau \wedge dt$. 

\subsubsection{Using the nondegeneracy hypothesis} We prepare for the proof of Lemma \ref{lem:derivatives} by writing down several consequences of the fact that $\phi^{\tau^0}$ is nondegenerate.

The spectrality property (Proposition \ref{prop:spectrality}) and the fact that $\phi^{\tau^0}$ is nondegenerate ensures that there is a finite, maximal set of orbit sets
$$\{\cO^1, \ldots, \cO^J\} \in \cP_{\bZ}(\phi^\tau)$$
and corresponding classes 
$$W^i \in H_2(M_\phi, \Theta_{\cO^i}, \Theta_{\text{ref}}; \bZ)$$
for each $i \in \{1, \ldots, J\}$ such that 
$$\bfA(\cO^i, W^i) = \int_{W^i} \omega_\phi + dH^{\tau^0} \wedge dt = c_\sigma(H^{\tau^0}; \bfS)$$
for each $i \in \{1, \ldots, J\}$. Since $\phi^{\tau^0}$ is nondegenerate, there is an open neighborhood $V$ of $\tau^0$ in $(0, 1)^N$ and a set of families of integral orbit sets
$$\{\{\cO^1_\tau\}_{\tau \in V}, \ldots, \{\cO^J_\tau\}_{\tau \in V}\}$$
satisfying the following properties:
\begin{itemize}
    \item $\cO^i_\tau \in \cP_{\bZ}(\phi^\tau)$ for every $\tau \in V$.
    \item $\cO^i_{\tau^0} = \cO^i$ for every $i \in \{1, \ldots, J\}$.
    \item $\cO^i_\tau$ is smooth as a function of $\tau \in V$ for every $i \in \{1, \ldots, J\}$. 
\end{itemize}

In the third point, the orbit sets $\cO^i_\tau$ are smooth as a function of $\tau \in V$ if they decompose as a sum
$$\cO^i_\tau = \sum_{j=1}^n a^i_j S^i_{j,\tau}$$
where the coefficients $a^i_j$ and periods $|S^i_{j,\tau}|$ are independent of $\tau$ and the orbits $S^i_{j,\tau}$ vary smoothly in $\tau$. Each family $\{\cO^i_\tau\}_{\tau \in V}$ has a unique associated family $\{W^i_\tau\}_{\tau \in V}$ satisfying the following properties:
\begin{itemize}
    \item $W^i_\tau \in H_2(M_\phi, \Theta_{\cO^i_\tau}, \Theta_{\text{ref}}; \bZ)$ for every $\tau \in V$.
    \item $W^i_{\tau^0} = W^i$ for every $i \in \{1, \ldots, J\}$. 
    \item $W^i_\tau$ is continuous as a function of $\tau \in V$ for every $i \in \{1, \ldots, J\}$ in the sense that for every closed $2$-form $\alpha$ on $M_\phi$, the function 
    $$\tau \mapsto \int_{W^i_\tau} \alpha$$
    on $V$ is continuous.
\end{itemize}

The class $W^i_{\tau^0 + \tau}$ is constructed as follows for any fixed $\tau$ in a neighborhood of the origin. For $s \in [0,1]$, the family of orbit sets $\{\Theta_{\cO^i_{\tau^0 + s\tau}}\}_{s \in [0,1]}$ can be regarded as a set of smooth immersed surfaces, denoted by $\Sigma^i_\tau$, in $M_\phi$ with positive integer multiplicities with boundary given by $\Theta_{\cO^i_{\tau^0 + \tau}} - \Theta_{\cO^i}$. Then $W^i_{\tau^0 + \tau}$ is the sum of $W^i$ and $\Sigma^i_\tau$. 

We now show that the spectral invariants $c_\sigma(H^\tau; \bfS)$ for $\tau$ near $\tau^0$ are each recovered by $\bfA(\cO^i_\tau, W^i_\tau)$ for some $i \in \{1, \ldots, J\}$. 

\begin{lem} \label{lem:spectrality}
There is some open neighborhood of $\tau_0$ in $(0,1)^N$ such that if $\tau$ lies in this neighborhood, then
$$c_\sigma(H^\tau; \bfS) \in \{\bfA(\cO^1_\tau, W^1_\tau), \ldots, \bfA(\cO^J_\tau, W^J_\tau)\}.$$
\end{lem}

\begin{proof}
Assume for the sake of contradiction that the conclusion of the lemma is false. Then there is a sequence $\tau_i \to \tau^0$ such that, for every $i$, 
$$c_\sigma(H^{\tau_i}; \bfS) \not\in \{\bfA(\cO^1_{\tau_i}, W^1_{\tau_i}), \ldots, \bfA(\cO^J_{\tau_i}, W^J_{\tau_i})\}.$$

Proposition \ref{prop:hoferContinuity} (Hofer continuity) implies that
$$\lim_{i \to \infty} c_\sigma(H^{\tau_i}; \bfS) = c_\sigma(H^{\tau^0}; \bfS).$$

Proposition \ref{prop:spectrality} (spectrality) implies that there is a sequence $\{(\cO_i, W_i)\}$ such that, for every $i$, $(\cO_i, W_i)$ is a PFH generator for $\phi^{\tau_i}$ and $\bfA(\cO_i, W_i) = c_\sigma(H^{\tau_i}; \bfS)$. 

A compactness argument shows, after passing to a subsequence, there is a PFH generator $(\cO, W)$ for $\phi^{\tau^0}$ such that $\cO_i \to \cO$ and $W_i \to W$. The former convergence is convergence as formal linear combinations of ordered sets of points in $\Sigma$, and the latter convergence is convergence of functionals on closed $2$-forms. We conclude that
$$\bfA(\cO, W) = \lim_{i \to \infty} \bfA(\cO_i, W_i) = c_\sigma(H^{\tau^0}; \bfS).$$

Therefore, there is some $j \in \{1, \ldots, J\}$ such that $\cO = \cO^j$ and $W = W^j$. Since $\phi^{\tau^0}$ is nondegenerate, we conclude that for sufficiently large $i$, we must have that $\cO_i = \cO^j_{\tau_i}$ and $W_i = W^j_{\tau_i}$. This yields a contradiction, which proves the lemma. 
\end{proof}

\subsubsection{The computation} Retain the notation for the neighborhood $V$ of $\tau^0$ and the smooth families $(\cO^i_\tau)_{\tau \in V}$, $(W^i_\tau)_{\tau \in V}$ introduced above. Fix any vector $\tau \in \bR^N$ in the open ball of radius $\delta$ around the origin, where $\delta$ is the constant from Lemma \ref{lem:spectrality}. We will complete the proof of Lemma \ref{lem:derivatives} by showing that for any $i \in \{1, \ldots, J\}$, 
\begin{equation} \label{eq:derivatives1} \partial_s|_{s = 0} \bfA(\cO^i_{\tau^0 + s\tau}, W^i_{\tau^0 + s\tau}) = (\int_{\Theta_{\cO^i}} - \int_{\Theta_{\text{ref}}}) (\partial_s|_{s = 0}H^{\tau^0 + s\tau} dt) \end{equation}
for any choice of $\tau$. This completes the proof of Lemma \ref{lem:derivatives} since $c_\sigma(H^{\tau}; \bfS)$ is differentiable at $\tau^0$, and by Lemma \ref{lem:spectrality} it is equal to one of the actions $\bfA(\cO^i_{\tau}, W^i_\tau)$ for any $\tau$ near $\tau^0$. 

We use Stokes' theorem to compute the action:
\begin{equation} \label{eq:derivatives2}
\begin{split} 
\bfA(\cO^i_{\tau^0 + s\tau}, W^i_{\tau^0 + s\tau}) &= \int_{W^i_{\tau^0 + s\tau}} (\omega_\phi + dH^{\tau^0 + s\tau} \wedge dt) \\
&= \int_{W^i_{\tau^0 + s\tau}} \omega_\phi + dH^{\tau^0} \wedge dt + \int_{\Theta_{\cO^i_{\tau^0 + s\tau}}} (H^{\tau^0 + s\tau} - H^{\tau^0})dt \\
&\qquad - \int_{\Theta_{\text{ref}}} (H^{\tau^0 + s\tau} - H^{\tau^0})dt.
\end{split}
\end{equation}

We now break up the last line of (\ref{eq:derivatives2}) into pieces and compute the $s$-derivative on each piece. First, we compute
\begin{equation}
    \label{eq:derivatives3}
    \begin{split}
    & \partial_s|_{s = 0}(\int_{\Theta_{\cO^i_{\tau^0 + s\tau}}} H^{\tau^0 + s\tau} dt - \int_{\Theta_{\text{ref}}} H^{\tau^0 + s\tau} dt) - \partial_s|_{s = 0}\big(\int_{\Theta_{\cO^i_{\tau^0 + s\tau}}} H^{\tau^0} dt\big)\\
    &\qquad = (\int_{\Theta_{\cO^i}} - \int_{\Theta_{\text{ref}}}) (\partial_s|_{s = 0} H^{\tau^0 + s\tau} dt).  
    \end{split}
\end{equation}

Next, we compute
\begin{equation} \label{eq:derivatives4}
\begin{split}
\partial_s|_{s = 0} \int_{W^i_{\tau^0 + s\tau}} (\omega_\phi + dH^{\tau^0} \wedge dt) &= \partial_s|_{s = 0} (\int_{W^i} + \int_{\Sigma^i_{\tau^0 + s\tau}}) (\omega_\phi + dH^{\tau^0} \wedge dt)\\
&= \int_{\Theta_{\cO^i}} (\omega_\phi + dH^{\tau^0} \wedge dt)(\partial_s|_{s = 0}\Theta_{\cO^i_{\tau^0 + s\tau}}, -) \\
&= 0.
\end{split}
\end{equation}

The second to last line uses the fact that the surface $\Sigma^i_{\tau^0 + s\tau}$ is foliated smoothly by the orbit sets along the path $s \mapsto \tau^0 + s\tau$. The last line uses the fact that $\Theta_{\cO^i}$ is tangent to the vector field $R + X_{H^{\tau^0}}$, where $R$ is the canonical vector field for $\phi$ and $X_{H^{\tau^0}}$ is the Hamiltonian vector field for $H^{\tau^0}$, which contracts with the two-form $\omega_\phi + dH^{\tau^0} \wedge dt$ to zero. Taking the $s$-derivative of (\ref{eq:derivatives2}) and plugging in (\ref{eq:derivatives3}) and (\ref{eq:derivatives4}) yields (\ref{eq:derivatives1}) and proves the lemma. 

\bibliographystyle{alpha}
\bibliography{main}

\begin{thebibliography}{BSHSa21}

\bibitem[BSHSa21]{actionLinking}
David Bechara~Senior, Umberto~L. Hryniewicz, and Pedro A.~S. Salom\~{a}o.
\newblock On the relation between action and linking.
\newblock {\em J. Mod. Dyn.}, 17:319--336, 2021.

\bibitem[CDHR22]{CDHR22}
Vincent Colin, Pierre Dehornoy, Umberto Hryniewicz, and Ana Rechtman.
\newblock Generic properties of {$3$}-dimensional {R}eeb flows: {B}irkhoff
  sections and entropy.
\newblock {\em arXiv preprint arXiv:2202.01506}, 2022.

\bibitem[CGHS20]{simplicity}
Dan Cristofaro-Gardiner, Vincent Humili\`{e}re, and Sobhan Seyfaddini.
\newblock Proof of the simplicity conjecture.
\newblock {\em arXiv preprint arXiv:2001.01792}, 2020.

\bibitem[CGHS21]{largeScale}
Dan Cristofaro-Gardiner, Vincent Humili\`{e}re, and Sobhan Seyfaddini.
\newblock {PFH} spectral invariants on the two-sphere and the large scale
  geometry of {H}ofer's metric.
\newblock {\em arXiv preprint arXiv:2102.04404}, 2021.

\bibitem[CGPZ21]{closingLemma21}
Dan Cristofaro-Gardiner, Rohil Prasad, and Boyu Zhang.
\newblock Periodic {F}loer homology and the smooth closing lemma for
  area-preserving surface diffeomorphisms.
\newblock {\em arXiv preprint arXiv:2110.02925}, 2021.

\bibitem[EH21]{edtmairHutchings}
Oliver Edtmair and Michael Hutchings.
\newblock {PFH} spectral invariants and {$C^\infty$} closing lemmas.
\newblock {\em arXiv preprint arXiv:2110.02463}, 2021.

\bibitem[EY21]{edtmairYao21}
Oliver Edtmair and Yuan Yao.
\newblock Equidistributed periodic orbits of {$C^\infty$}-generic
  area-preserving surface diffeomorphisms.
\newblock {\em in preparation}, 2021.

\bibitem[Her79]{herman}
Michael-Robert Herman.
\newblock Sur la conjugaison diff\'{e}rentiable des diff\'{e}omorphismes du
  cercle \`a des rotations.
\newblock {\em Inst. Hautes \'{E}tudes Sci. Publ. Math.}, (49):5--233, 1979.

\bibitem[Hut02]{HutchingsECH02}
Michael Hutchings.
\newblock An index inequality for embedded pseudoholomorphic curves in
  symplectizations.
\newblock {\em J. Eur. Math. Soc. (JEMS)}, 4(4):313--361, 2002.

\bibitem[Iri21]{Irie21}
Kei Irie.
\newblock Equidistributed periodic orbits of {$C^\infty$}-generic
  three-dimensional {R}eeb flows.
\newblock {\em J. Symplectic Geom.}, 19(3):531--566, 2021.

\bibitem[MNS19]{MarquesNevesSong}
Fernando~C. Marques, Andr\'{e} Neves, and Antoine Song.
\newblock Equidistribution of minimal hypersurfaces for generic metrics.
\newblock {\em Invent. Math.}, 216(2):421--443, 2019.

\end{thebibliography}

\end{document}